\documentclass[reqno]{amsart}
\usepackage{amsmath,amsthm,amsfonts,amssymb,amscd}
\usepackage{enumerate}
\usepackage[vcentermath,enableskew]{youngtab}
\usepackage{microtype}
\usepackage{tikz}
\usepackage{mathrsfs}
\usepackage[ocgcolorlinks, linkcolor=red!80!black, citecolor=blue!80!black]{hyperref}
\usepackage{mathtools}

\newtheorem{thm}{Theorem}[section]
\newtheorem*{thm*}{Theorem}
\newtheorem{lem}[thm]{Lemma}
\theoremstyle{definition} \newtheorem{defn}[thm]{Definition}
\newtheorem*{defn*}{Definition}
\theoremstyle{definition} \newtheorem{ex}[thm]{Example}
\newtheorem*{lem*}{Lemma}
\newtheorem{cor}[thm]{Corollary}
\newtheorem*{cor*}{Corollary}
\theoremstyle{definition} \newtheorem{rem}[thm]{Remark}

\newtheorem*{conj*}{Conjecture}
\newtheorem{conj}{Conjecture}

\newtheorem{prop}[thm]{Proposition}
\newtheorem{exlem}[thm]{Example-Lemma}

\newcommand{\CC}{\mathbb{C}}
\newcommand{\NN}{\mathbb{N}}
\newcommand{\ZZ}{\mathbb{Z}}
\newcommand{\RR}{\mathbb{R}}

\DeclareMathOperator{\tr}{tr}

\DeclareMathOperator{\ch}{ch}
\DeclareMathOperator{\cyc}{cyc}
\DeclareMathOperator{\supp}{supp}

\DeclareMathOperator{\rank}{rank}

\DeclareMathOperator{\type}{type}
\DeclareMathOperator{\sgn}{sgn}

\renewcommand{\epsilon}{\varepsilon}

\newcommand{\eqdef}{\overset{\text{def}}{=}}

\begin{document}

\author{Brendan Pawlowski}
\title{Chromatic symmetric functions via the group algebra of $S_n$}

\begin{abstract}
We prove some Schur positivity results for the chromatic symmetric function $X_G$ of a (hyper)graph $G$, using connections to the group algebra of the symmetric group. The first such connection works for (hyper)forests $F$: we describe the Schur coefficients of $X_F$ in terms of eigenvalues of a product of Hermitian idempotents in the group algebra, one factor for each edge (a more general formula of similar shape holds for all chordal graphs). Our main application of this technique is to prove a conjecture of Taylor on the Schur positivity of certain $X_F$, which implies Schur positivity of the \emph{formal group laws} associated to various combinatorial generating functions. We also introduce the \emph{pointed chromatic symmetric function} $X_{G,v}$ associated to a rooted graph $(G,v)$. We prove that if $X_{G,v}$ and $X_{H,w}$ are positive in the \emph{generalized Schur basis} of Strahov, then the chromatic symmetric function of the wedge sum of $(G,v)$ and $(H,w)$ is Schur positive.
\end{abstract}

\maketitle

\section{Introduction}

Let $G$ be a finite simple graph with vertices $V$ and edges $E$. A \emph{coloring} of $G$ is a function $\kappa : V \to \NN$, and $\kappa$ is \emph{proper} if $(v,w) \in E$ implies $\kappa(v) \neq \kappa(w)$.

\begin{defn}[\cite{stanley-chromatic}] The \emph{chromatic symmetric function} of $G$ is the formal power series
\begin{equation*}
X_G = \sum_{\kappa} x_{\kappa},  
\end{equation*}
where $\kappa$ runs over proper colorings of $G$ and $x_{\kappa} = \prod_{v \in V} x_{\kappa(v)}$.
\end{defn} \vspace{-0.3cm}
The more classical chromatic polynomial $\chi_G(k)$ can be recovered as $X_G(\overbrace{1, \ldots, 1}^k, 0, 0, \ldots)$. Being a symmetric function, $X_G$ is a linear combination of Schur functions, and the question we are concerned with here is when these coefficients are nonnegative, i.e. when $X_G$ is \emph{Schur positive}. We will often say ``$G$ is Schur positive'' to mean that $X_G$ is Schur positive.  

The best-known result of this type is due to Gasharov: if $P$ is a $(\mathbf{3}+\mathbf{1})$-free poset (does not contain the disjoint union of a $3$-chain and a $1$-chain as a subposet), and $G$ is its incomparability graph (the graph on $P$ with an edge $(x,y)$ whenever $x$ and $y$ are incomparable in $P$), then $X_G$ is Schur positive \cite{gasharov-chromatic}. Stanley and Stembridge conjectured a stronger claim, that such $X_G$ are positive in the basis of elementary symmetric functions, and this conjecture remains open \cite{stanley-stembridge-immanents}.

Let $\Lambda$ be the ring of symmetric functions over $\CC$, and $\CC[S_n]$ be the complex group algebra of the symmetric group $S_n$. Write $\cyc(\sigma)$ for the cycle type of $\sigma \in S_n$, i.e. the partition of $n$ whose parts are the lengths of the cycles of $\sigma$.
\begin{defn}
    The \emph{Frobenius characteristic map} $\ch : \CC[S_n] \to \Lambda$ is the linear map with $\sigma \mapsto \frac{1}{n!}p_{\cyc(\sigma)}$, where $p_{\lambda}$ is a power sum symmetric function.
\end{defn}
Here is our first tool for proving Schur positivity.

\begin{lem}[cf.  Lemma~\ref{lem:trace}] \label{lem:schur-positive} Let $\alpha \in \CC[S_n]$, viewed as the operator $\CC[S_n] \to \CC[S_n]$, $x \mapsto \alpha x$. If $\alpha$ acts with nonnegative trace on the irreducible submodules of $\CC[S_n]$, then $\ch(\alpha)$ is Schur positive. \end{lem}
Lemma~\ref{lem:schur-positive} is most familiar in the case when $\alpha$ is central. The Frobenius characteristic $\ch : Z(\CC[S_n]) \to \Lambda$ is a linear isomorphism sending the characters of $S_n$ to the Schur functions of degree $n$, and the eigenvalues of $\alpha \in Z(\CC[S_n])$ are essentially the same as the Schur coefficients of $\ch(\alpha)$. We will apply Lemma~\ref{lem:schur-positive} to $\alpha$ which are usually not central, but which have nice factorizations that sometimes let us deduce nonnegativity of their eigenvalues using standard linear algebra techniques.

In particular, given a forest $F$ with vertex set $[n] \eqdef \{1, 2, \ldots, n\}$ and edge set $E$, define the operator
\begin{equation*}
\alpha_{F} = n!\prod_{(i,j) \in E} (1 - (i\,j)) \in S_n,
\end{equation*}
where the product is taken in some unspecified linear order.

\begin{thm}[cf. Theorem~\ref{thm:chromatic-operator}] \label{thm:forest-operator} Let $F$ be a forest. Then $\ch(\alpha_{F}) = X_F$, regardless of the choice of edge ordering used to define $\alpha_F$. \end{thm}
This theorem can be generalized: in Section~\ref{sec:chromatic-operators} we will see that there is a partition $E_1, \ldots, E_q$ of the edges of any chordal graph $G$ with the property that
\begin{equation*}
    \alpha_G = n!\prod_{p=1}^q \left(1 - \sum_{(i,j) \in E_p} (i,j)\right)
\end{equation*}
maps to $X_G$ under $\ch$.

\begin{ex} \label{ex:path}
The operator $1 - (i\,j)$ is positive semidefinite with respect to the Hermitian inner product on $\CC[S_n]$ having $S_n$ as an orthonormal basis. Thus, if $P_3$ is the path with edges $(1,2),(2,3)$, then $\alpha_{P_3} = (1 - (1\,2))(1 - (2\,3))$ is the product of two positive semidefinite operators, and therefore has nonnegative eigenvalues \cite[Corollary 7.6.2]{horn-johnson}. It follows from Lemma~\ref{lem:schur-positive} and  Theorem~\ref{thm:forest-operator} that $P_3$ is Schur positive.
\end{ex}
Our main application of this technique will be to prove a conjecture of Taylor \cite{jair-thesis} on the Schur positivity of certain path-like hypergraphs $G$ (Section~\ref{sec:hyperforests}).  Taylor showed that this conjecture would imply the Schur positivity of the \emph{formal group law} $f(f^{-1}(x_1) + f^{-1}(x_2) + \cdots)$ associated to the generating function $f(x) = \sum_n a_n x^n$ of various interesting families of combinatorial objects.

The \emph{wedge sum} $G \vee H$ of two graphs $G$ and $H$ with distinguished vertices $v$ and $w$ is their disjoint union modulo the identification of the two distinguished vertices. In Section~\ref{sec:pointed-chrom-sym} we define the \emph{pointed chromatic symmetric function} $X_{G,v} \in \Lambda[t]$, and show it has the following properties.
\begin{itemize}
\item The $\Lambda$-linear function defined by $t^i \mapsto p_{i+1}$ sends $X_{G,v}$ to $X_G$.
\item $X_{G,v}$ satisfies a deletion-contraction recurrence.
\item $X_{G,v}X_{H,w} = X_{G \vee H, v}$.
\item If $X_{G,v}$ and $X_{H,w}$ expand nonnegatively in the \emph{pointed Schur basis} of $\Lambda[t]$ (called the \emph{generalized Schur basis} in \cite{strahov}), then $G \vee H$ is Schur positive.
\end{itemize}

\begin{ex}
We will show that $X_{P_n, 1}$ is pointed Schur positive, as is $X_{C, v}$ where $C$ is a cycle graph. Thus any graph
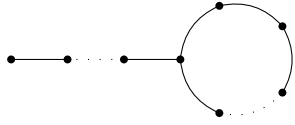
\begin{center}
    \begin{tikzpicture}[scale=0.75, every node/.style = {circle, fill, inner sep=0, minimum size=3pt}]
        \draw (-1,0) node {} to[bend left] (-0.309, 0.951) node {} to[bend left] (0.809, 0.588) node {} to[bend left] (0.809, -0.588) node {};
        \draw[loosely dotted]  (0.809, -0.588) to[bend left] (-0.309, -0.951) node {};
        \draw (-0.309, -0.951) to[bend left] (-1,0);
        \draw (-1,0) to (-2,0) node {};
        \draw[loosely dotted] (-2,0) to (-3,0);
        \draw (-3,0) node {} to (-4,0) node {};
    \end{tikzpicture}
\end{center}
is Schur positive.
\end{ex}
In Section~\ref{sec:pointed-chrom-sym-path} we investigate the expansion of $X_{G,v}$ in a pointed analogue of the elementary symmetric functions. Pointed $e$-positivity of $X_{G,v}$ implies $e$-positivity of $X_G$, and we describe a version of the Stanley-Stembridge conjecture for pointed chromatic symmetric functions.

Our analysis of pointed chromatic symmetric functions will again rely on representation theory. Up to predictable positive scalars, the eigenvalues of some $\beta$ in the center $Z(\CC[S_n])$ of $\CC[S_n]$ are also the coefficients of $\beta$ in the basis of characters, or the Schur coefficients of $\ch(\beta)$. Now let $Z_{\CC[S_n]}(\CC[S_{n-1}])$ denote the centralizer of $\CC[S_{n-1}]$ in $\CC[S_n]$. The algebra $Z_{\CC[S_n]}(\CC[S_{n-1}])$ turns out to be commutative and semisimple, so by Wedderburn's theorem it again has a distinguished basis of ``generalized characters''. Strahov \cite{strahov} defined an analogue $\ch' : Z_{\CC[S_n]}(\CC[S_{n-1}]) \to \Lambda[t]$ of the Frobenius characteristic, and defined the \emph{pointed Schur functions} mentioned above as the images of the generalized characters under $\ch'$. The pointed chromatic symmetric function $X_{F,v}$ is $\ch'(\beta)$ for a certain $\beta \in Z_{\CC[S_n]}(\CC[S_{n-1}])$, and the eigenvalues of $\beta$ are essentially the pointed Schur coefficients of $\ch'(\beta)$.

\section{Chromatic operators} \label{sec:chromatic-operators}

\subsection{Forests} \label{subsec:chromatic-operators-forests}

The following lemma of D\'enes provides the basic connection between trees and permutation factorization which we will exploit.
\begin{lem}[\cite{denes}] \label{lem:tree-cycle-type} Let $G$ be a graph with vertex set $[n]$. Then
\begin{equation*}
\prod_{(i,j) \in E(G)} (i\,\,j) \in S_n,
\end{equation*}
the product being taken in any order, is an $n$-cycle if and only if $G$ is a tree.
\end{lem}
\begin{proof} See \cite[\S 2]{moszkowski}, or Lemma~\ref{lem:hyperforest-operator-ch} below. \end{proof}

Write $\type(G)$ for the partition whose parts are the sizes (number of vertices) of the connected components of $G$. Recall that $\cyc(\sigma)$ is the partition consisting of the lengths of the cycles of $\sigma$.
\begin{cor} \label{cor:forest-cycle-type} If $F$ is a forest, then $\cyc\left(\prod_{(i,j) \in E(F)} (i\,\,j)\right) = \type(F)$. \end{cor}

Stanley showed how to expand $X_G$ in the power sum basis of $\Lambda$. Given $S \subseteq E(G)$, let $G_S$ be the graph with vertex set $V(G)$ and edge set $S$.
\begin{thm}[\cite{stanley-chromatic}, Theorem 2.5] \label{thm:power-sum-expansion}
$\displaystyle X_G = \sum_{S \subseteq E(G)} (-1)^{|S|} p_{\type(G_S)}$.
\end{thm}

\begin{defn} \label{defn:chromatic-operator}
    Given a forest $F$ with vertex set $[n]$ and a linear ordering $\pi$ of its edges, define the associated \emph{chromatic operator}
    \begin{equation*}
        \alpha_{F,\pi} = n!\hspace{-7pt}\prod_{(i,j) \in E(F)}\hspace{-7pt} (1-(i\,\,j)) \in \CC[S_n],
    \end{equation*}
    where the product is taken in the order prescribed by $\pi$.
\end{defn}
The dependence on $\pi$ will often be unimportant, in which case we write $\alpha_F$. Corollary~\ref{cor:forest-cycle-type} and Theorem~\ref{thm:power-sum-expansion} immediately give:
\begin{thm} \label{thm:chromatic-operator} $\ch \alpha_{F,\pi} = X_F$ for any edge ordering $\pi$ of a forest $F$.
\end{thm}
If $G$ is not a forest, then $\ch \alpha_{G,\pi}$ usually does not equal $X_G$, and may depend on the choice of $\pi$. For instance, if $G$ is a length 4 cycle with the two edge orderings $\pi = ((1,2),(2,3),(3,4),(4,1))$ and $\pi' = ((1,2),(3,4),(2,3),(4,1))$, then $\ch \alpha_{G,\pi} \neq \ch \alpha_{G,\pi'}$.

We now recall a little representation theory of finite groups. Let $\Gamma$ be a finite group. The center $Z(\CC[\Gamma])$ consists of the conjugation-invariant elements of $\CC[\Gamma]$, so has a natural basis $\{\sum_{g \in C} g\}$ where $C$ runs over the conjugacy classes of $\Gamma$. On the other hand, because $\CC[\Gamma]$ is a semisimple ring, Wedderburn's theorem says that $Z(\CC[\Gamma])$ has a unique (up to ordering) basis of idempotents adding to $1$, which are necessarily orthogonal (distinct idempotents multiply to $0$). Explicitly, these idempotents are $\epsilon_{\chi} \eqdef \frac{\chi(1)}{|\Gamma|} \sum_g \overline{\chi(g)}g$ where $\chi$ runs over the irreducible characters of $\Gamma$. The group algebra $\CC[\Gamma]$ decomposes as a direct sum $\bigoplus_\chi (V^{\chi})^{\oplus \chi(1)}$, where $V^{\chi}$ is the irreducible module with character $\chi$, and multiplication by the idempotent $\epsilon_{\chi}$ is the unique $\Gamma$-equivariant projection onto $(V^{\chi})^{\oplus \chi(1)}$.

In the case $\Gamma = S_n$, symmetric functions appear because the transition matrix between the conjugacy class basis and idempotent basis described above is the transition matrix between rescalings of the power sum basis and Schur basis. More specifically, the Frobenius characteristic $\ch : Z(\CC[S_n]) \to \Lambda_n$ is a linear isomorphism sending $\sum_{\sigma \in S_n} \chi^{\lambda}(\sigma)\sigma = \frac{n!}{\chi^{\lambda}(1)}\epsilon_{\lambda}$ to $s_{\lambda}$; here and below we abbreviate $\epsilon_{\chi^\lambda}$ to $\epsilon_\lambda$ and $\cramped{V^{\chi^{\lambda}}}$ to $V^\lambda$.

Write $[s_{\lambda}]f$ for the coefficient of $s_{\lambda}$ in $f \in \Lambda$. One can deduce from basic character identities that $[s_{\lambda}]p_{\mu} = \chi^{\lambda}(\mu)$ \cite[I.7]{macdonald}.
\begin{lem} \label{lem:trace} The trace of $\alpha \in \CC[S_n]$ acting on $V^\lambda$ is $\tr(\alpha \epsilon_\lambda) = n![s_{\lambda}]\ch(\alpha)$. \end{lem}
    \begin{proof}
        As described above, $\epsilon_{\lambda}$ is a projection onto the irreducible representation $V^{\lambda}$, so $\tr(\alpha|_{V^\lambda}) = \tr(\alpha \epsilon_\lambda)$. To see that this equals $n![s_{\lambda}]\ch(\alpha)$, it suffices by linearity to assume $\alpha \in S_n$. Now
        \begin{equation*}
            \tr(\alpha|_{V^{\lambda}}) =  \chi^\lambda(\alpha) =  [s_{\lambda}]p_{\cyc(\alpha)} = n![s_{\lambda}]\ch(\alpha).
        \end{equation*}
    \end{proof}

\begin{cor} \label{cor:eigenvalues}
If $\alpha \in \RR[S_n]$ is positive semistable (all eigenvalues have nonnegative real part), then $\ch \alpha$ is Schur positive.
\end{cor}
\begin{proof}
 The irreducible representations of $S_n$ are defined over $\RR$ and $\alpha \in \RR[S_n]$, so  the eigenvalues of $\alpha|_{V^{\lambda}}$ come in complex conjugate pairs. By the positive semistability assumption, their sum is nonnegative, and this sum is $\tr(\alpha|_{V^\lambda}) = n![s_{\lambda}]\ch(\alpha)$.
\end{proof}

For our purposes, Theorem~\ref{thm:chromatic-operator} allows us to work with any convenient edge ordering. We conjecture that the spectrum of $\alpha_{F,\pi}$ does not depend on $\pi$, and that there is a nice family of forests with positive semistable chromatic operators.
\begin{conj} \label{conj:forest-conjecture} Let $F$ be a forest.
\begin{enumerate}[(a)]
\item The characteristic polynomial of $\alpha_{F,\pi}$ does not depend on the edge ordering $\pi$.
\item The number of forests on $n$ vertices whose chromatic operator is positive semistable is $f_{n-1}$, where $f_n$ is the $n$th Fibonacci number (using the convention that $f_1 = f_2 = 1$).
\end{enumerate}
\end{conj}

The eigenvalues of $\alpha_{F \sqcup G}$ are the disjoint union of the eigenvalues of $\alpha_F$ and of $\alpha_G$, so $\alpha_{F}$ is positive semistable if and only if $\alpha_C$ is for all connected components $C$ of $F$. Figure~\ref{fig:conjecture-trees} shows the other trees $T$ on at most $10$ vertices for which $\alpha_T$ is positive semistable, omitting the paths for brevity ($\alpha_P$ is positive semistable for any path $P$, as per the proof of Theorem~\ref{thm:edge-bipartite-hypergraphs}).

\begin{figure} \label{fig:conjecture-trees} \caption{The trees with at most 10 vertices whose chromatic operator is positive semistable, not including the paths.}
    \vspace{5mm}
    \begin{tikzpicture}[scale=0.25, every node/.style={circle, black, fill, inner sep=0pt, minimum size=3pt}]
        \node at (0,0) {}
            child { node {}
                child { node {}
                    child { node {} }
                    child { node {} }
                }
            };

        \node at (8,0) {}
            child { node {}
                child { node {}
                    child { node {} }
                    child { node {}
                        child { node {} } }
                }
            };   

        \node at (16,0) {}
            child { node {}
                child { node {}
                    child { node {} }
                    child { node {}
                        child { node {}
                            child { node {} } } }
                }
            };

        \node at (20,0) {}
            child { node {}
                child { node {}
                    child { node {}
                        child { node {} } }
                    child { node {}
                        child { node {} } }
                }
            };

        \node at (24,0) {}
            child { node {}
                child { node {}
                    child { node {} }
                    child { node {} }
                    child { node {} 
                        child { node {} }
                    }
                }
            };

        \node at (32,0) {}
            child { node {}
                child { node {}
                    child { node {} }
                    child { node {} 
                        child { node {}
                            child { node {}
                                child { node {} }
                            }
                        }
                    }
                }
            };

        \node at (36,0) {}
            child { node {}
                child { node {}
                    child { node {} }
                    child { node {} 
                        child { node {} }
                        child { node {}
                            child { node {} }
                        }
                    }
                }
            };

        \node at (40,0) {}
            child { node {}
                child { node {} 
                    child { node {}
                        child { node {} }
                    }
                    child { node {}
                        child { node {}
                            child { node {} }
                        }
                    }
                }
            };

        \node at (44,0) {}
            child { node {}
                child { node {}
                    child { node {} }
                    child { node {}
                        child { node {} }
                    }
                    child { node {}
                        child { node {} }
                    }
                }
            };

        \node at (7,-11) {}
            child { node {}
                child { node {}
                    child { node {} }
                    child { node {}
                        child { node {}
                            child { node {}
                                child { node {}
                                    child { node {} }
                                }
                            }
                        }
                    }
                }
            };

        \node at (12,-11) {}
            child { node {}
                child { node {}
                    child { node {} 
                        child { node {} }
                    }
                    child { node {}
                        child { node {}
                            child { node {}
                                child { node {} }
                            }
                        }
                    }
                }
            };  
        
        \node at (17,-11) {}
            child { node {} 
                child { node {}
                    child { node {} }
                    child { node {}
                        child { node {}
                            child { node {} }
                            child { node {}
                                child { node {} }
                            }
                        }
                    }
                }
            };

        \node at (22,-11) {}
            child { node {}
                child { node {}
                    child { node {} }
                    child { node {}
                        child { node {}
                            child { node {} }
                        }
                        child { node {}
                            child { node {} }
                        }
                    }
                }
            };

        \node at (27,-11) {}
            child { node {}
                child { node {}
                    child { node {}
                        child { node {}
                            child { node {} }
                        }
                    }
                    child { node {}
                        child { node {}
                            child { node {} }
                        }
                    }
                }
            };

        \node at (32,-11) {}
            child { node {}
                child { node {}
                    child { node {} }
                    child { node {}
                        child { node {} }
                    }
                    child { node {}
                        child { node {}
                            child { node {} }
                        }
                    }
                }
            };

        \node at (37,-11) {}
            child { node {}
                child { node {}
                    child { node {}
                        child { node {} }
                    }
                    child { node {}
                        child { node {} }
                    }
                    child { node {}
                        child { node {} }
                    }
                }
            };

        \node at (0,-25) {}
            child { node {}
                child { node {}
                    child { node {} }
                    child { node {}
                        child { node {} 
                            child { node {}
                                child { node {}
                                    child { node {}
                                        child { node {} }
                                    }
                                }
                            }
                        }
                    }
                }
            };

        \node at (5,-25) {}
            child { node {}
                child { node {}
                    child { node {} }
                    child { node {}
                        child { node {}
                            child { node {}
                                child { node {} }
                                child { node {}
                                    child { node {} }
                                }
                            }
                        }
                    }
                }
            };

        \node at (10,-25) {}
            child { node {}
                child { node {}
                    child { node {}
                        child { node {} }
                    }
                    child { node {}
                        child { node {}
                            child { node {}
                                child { node {}
                                    child { node {} }
                                }
                            }
                        }
                    }
                }
            };

        \node at (15,-25) {}
            child { node {}
                child { node {}
                    child { node {}
                        child { node {}
                            child { node {} }
                        }
                    }
                    child { node {}
                        child { node {}
                            child { node {}
                                child { node {} }
                            }
                        }
                    }
                }
            };

        \node at (20,-25) {}
            child { node {}
                child { node {}
                    child { node {} }
                    child { node {}
                        child { node {} }
                    }
                    child { node {}
                        child { node {}
                            child { node {}
                                child { node {} }
                            }
                        }
                    }
                }
            };

        \node at (25,-25) {}
            child { node {}
                child { node {}
                    child { node {}
                        child { node {} }
                    }
                    child { node {}
                        child { node {}
                            child { node {} }
                            child { node {}
                                child { node {} }
                            }
                        }
                    }
                }
            };

        \node at (30,-25) {}
            child { node {}
                child { node {}
                    child { node {}
                        child { node {} }
                    }
                    child { node {}
                        child { node {}
                            child { node {} }
                        }
                        child { node {}
                            child { node {} }
                        }
                    }
                }
            };

        \node at (35,-25) {}
            child { node {}
                child { node {}
                    child { node {} }
                    child { node {}
                        child { node {}
                            child { node {} }
                        }
                        child { node {}
                            child { node {}
                                child { node {} }
                            }
                        }
                    }
                }
            };

        \node at (40,-25) {}
            child { node {}
                child { node {}
                    child { node {} }
                    child { node {}
                        child { node {} }
                    }
                    child { node {}
                        child { node {}
                            child { node {} }
                        }
                        child { node {} }
                    }
                }
            };

        \node at (46,-25) {}
            child { node {}
                child { node {}
                    child { node {}
                        child { node {} }
                    }
                    child { node {}
                        child { node {} }
                    }
                    child { node {}
                        child { node {}
                            child { node {} }
                        }
                    }
                }
            };
    \end{tikzpicture}
\end{figure}
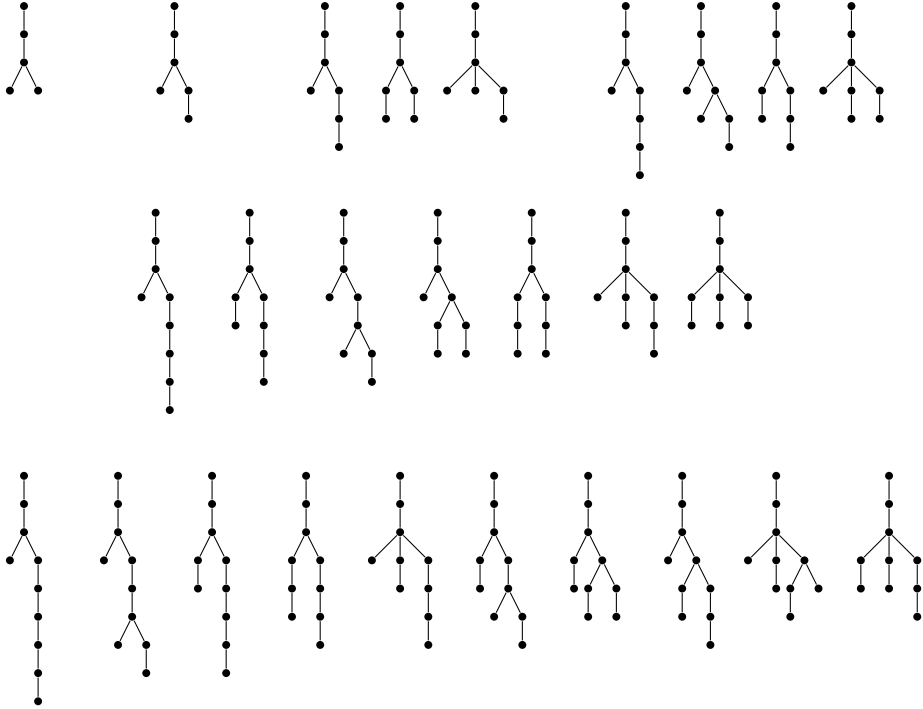

We have checked part (b) of Conjecture~\ref{conj:forest-conjecture} up to $10$ vertices. In general the operators $\alpha_{F,\pi}$ for different $\pi$ are \emph{not} similar to each other, and need not have the same minimal polynomial. For instance, if $F = P_4$ is the length 4 path with two edge orderings $\pi = ((1,2),(2,3),(3,4))$ and $\pi' = ((1,2),(3,4),(2,3))$, then $\alpha_{P_3, \pi'}$ is diagonalizable while $\alpha_{P_3, \pi}$ is not.

For $\alpha \in \CC[S_n]$, let $[1]\alpha$ denote the coefficient of the identity permutation in $\alpha$. Conjecture~\ref{conj:forest-conjecture}(a) would follow from the more combinatorial conjecture that $[1]\alpha_{F,\pi}^k$ is independent of $\pi$ for every positive integer $k$ (or just $k \leq n!$). Indeed, for $\alpha \in \CC[S_n]$ one has $\tr(\alpha) = n! [1]\alpha$, so the numbers $n! [1]\alpha_{F,\pi}^k$ are the power sum symmetric functions $p_k$ evaluated on the eigenvalues of $\alpha_{F,\pi}$. These numbers determine the elementary symmetric functions $e_k$ evaluated on the eigenvalues of $\alpha_{F,\pi}$, which are the coefficients of the characteristic polynomial of $\alpha_{F,\pi}$ (up to a predictable sign). This combinatorial conjecture holds when $k = 1, 2$: Lemma~\ref{lem:tree-cycle-type} implies that $\frac{1}{n!}[1]\alpha_{F,\pi} = 1$ and that $\frac{1}{n!}[1]\alpha_{F,\pi}^2$ is the number of matchings of the forest $F$. It appears that more generally, the symmetric function $\ch(\alpha_{F,\pi}^k)$ is independent of $\pi$, although it need not be monomial-positive for $k > 1$.

\subsection{Chordal graphs} \label{sec:chordal}
In this section we note that the product definition of $\alpha_F$ (Definition~\ref{defn:chromatic-operator}) can be generalized to a broader class of graphs. Put a total ordering on the edges of $G$. A \emph{circuit} of $G$ is the set of edges in a cycle, and a \emph{broken circuit} is a circuit minus the edge with the largest label. The \emph{broken circuit complex} $B_G$ of $G$ is the collection of all sets of edges which contain no  broken circuit of $G$.

\begin{defn} The \emph{join} of two simplicial complexes $\Delta_1$ and $\Delta_2$ is the complex $\Delta_1 \ast \Delta_2 = \{F_1 \cup F_2 : F_1 \in \Delta_1, F_2 \in \Delta_2\}$. A complex $\Delta$ \emph{factors completely} if it is the join of zero-dimensional complexes. \end{defn}

    \begin{ex} Take $G$ to be the complete graph $K_3$, with edges $a < b < c$. There is one broken circuit, namely $\{a,b\}$, and the broken circuit complex $B_{K_3}$ is the simplicial complex with groundset $\{a,b,c\}$ and facets $\{a,c\}, \{b,c\}$. This complex factors completely: it is the join $\{\{c\}\} \ast \{\{a\},\{b\}\}$.
    \end{ex}

Stanley gave the following alternative formula for $X_G$ in terms of power sums.
\begin{thm}[\cite{stanley-chromatic}, Theorem 2.9]
    $X_G = \sum_{S \in B_G} (-1)^{|S|} p_{\type(G|_S)}$.
\end{thm}

\begin{cor} If $B_G$ factors completely as $\Delta_1 \ast \cdots \ast \Delta_q$, define
    \begin{equation*}
        \alpha_G = n! \prod_{p=1}^q \left(1 - \sum_{(i,j) \in \Delta_p} (i\,j)\right).
    \end{equation*}
    Then $\ch(\alpha_G) = X_G$.
\end{cor}

\begin{proof}
    If $S \in B_G$, then $G|_S$ is a forest, given that if $S$ contained a circuit it would contain a broken circuit. Corollary~\ref{cor:forest-cycle-type} then implies that
    \begin{equation*}
        \type(G|_S) = \cyc\left(\prod_{(i,j) \in S} (i\,j)\right).
    \end{equation*}
    Thus $\ch$ maps $n!\sum_{S \in B_G} (-1)^{|S|} \prod_{(i,j) \in S} (i\,\,j)$ to $X_G$. But this sum factors into the form $\alpha_G$ if $B_G$ factors completely. 
\end{proof}

\begin{ex} Since $B_{K_3}$ factors completely as $\{\{(2,3)\}\} \ast \{\{(1,2)\}, \{(1,3)\}\}$, we can write $X_{K_3}$ as the image under $\ch$ of 
    \begin{equation*}
        6(1 - (2\,3))(1 - (1\,2) - (1\,3)).
    \end{equation*}
\end{ex}

Finally, there is a nice characterization of the graphs whose broken circuit complex factors completely.
\begin{thm}[\cite{broken-circuit-complexes}] The broken circuit complex $B_G$ factors completely if and only if $G$ is \emph{chordal}, i.e. has no induced cycles of size larger than $3$. \end{thm}

\subsection{Hyperforests} \label{sec:hyperforests}
\begin{defn} A \emph{hypergraph} is a pair $(V,E)$ where $V$ is a set and $E \subseteq 2^V \setminus \{\emptyset\}$. \end{defn}
We think of hypergraphs as graphs whose edges may contain more than two vertices.
\begin{defn}[\cite{stanley-chromatic-2}, \S3.3] A \emph{coloring} of a hypergraph $G = (V,E)$ is again a function $\kappa : V \to \NN$, and $\kappa$ is \emph{proper} if $|\kappa(e)| > 1$ for all $e \in E$ with $|e| > 1$. That is, there are no monochromatic edges under $\kappa$ except perhaps singletons. The \emph{chromatic symmetric function} of $G$ is again $X_G = \sum_{\kappa} x_\kappa$ where $\kappa$ runs over proper colorings of $G$.
\end{defn}

As pointed out in \cite{stanley-chromatic-2}, one might at first want to a define a proper coloring of a hypergraph $G$ to be proper if $\kappa(v) \neq \kappa(w)$ whenever $v \neq w$ are contained in a common edge, but this would lead to nothing new: such colorings are simply proper colorings of the graph on $V$ with an edge $(v,w)$ whenever $v, w$ are contained in a common edge of $G$.

\begin{ex} If $G$ is the hypergraph with vertices $[n]$ and a single edge $[n]$, then $X_G = p_1^n - p_n$. This is not the chromatic symmetric function of any graph: such a graph would have to have all non-constant colorings proper, hence no edges, but then the constant colorings would also be proper. \end{ex}

\begin{defn} The \emph{line graph} of a hypergraph $G$ is the graph $L(G)$ with vertex set $E(G)$, and an edge $(e,e')$ if and only if $e \cap e' \neq \emptyset$ for distinct edges $e, e' \in E(G)$. Say $G$ is \emph{edge $k$-colorable} if $L(G)$ admits a proper $k$-coloring. \end{defn}

\begin{ex} Suppose $G$ is a \emph{linear interval hypergraph}, meaning that $L(G)$ is a disjoint union of paths and if $e, e' \in E(G)$ are distinct, then $|e \cap e'| \leq 1$. Then $G$ is edge 2-colorable. By contrast, the graph $G$ with edges $\{1,2\}, \{1,3\}, \{1,4\}$ has $L(G) = K_3$, and is not edge 2-colorable.
\end{ex}

Taylor conjectured that linear interval hypergraphs are Schur positive \cite[Conjecture B]{jair-hypertree-chrom-sym}. In this section we prove Taylor's conjecture, and more generally that edge 2-colorable hyperforests are Schur positive---the hypergraph with edges $\{1,2,3\}, \{1,1'\}, \{2,2'\}, \{3,3'\}$ being an example of an edge 2-colorable hyperforest which is not a linear interval hypergraph.  Once we check that the machinery of \S\ref{subsec:chromatic-operators-forests} still works for hypergraphs, the proof will be essentially the short argument of Example~\ref{ex:path}.
\begin{defn}
A \emph{path} in a hypergraph is a sequence $v_1, e_1, v_2, e_2, \ldots, v_k, e_k, v_{k+1}$ where the $v_i$ are vertices and the $e_i$ are edges with $v_i \in e_i\cap e_{i-1}$, all entries in the sequence are distinct except perhaps that $v_{k+1} = v_1$, and $k > 1$. If $v_1 = v_{k+1}$ then the path is a \emph{cycle}. A hypergraph with no cycles is a \emph{hyperforest}, and a connected hyperforest is a \emph{hypertree}. (A hypergraph $G$ is connected if there is no nontrivial disjoint union $V(G) = V_1 \cup V_2$ such that every edge is contained in either $V_1$ or $V_2$.)
\end{defn}

\begin{ex} The hypergraph $G$ with edges $e_1 = \{1,2,3\}$ and $e_2 = \{2,3,4\}$ is \emph{not} a hypertree even though $L(G)$ is a tree, because it contains the cycle $2, e_1, 3, e_2, 2$. \end{ex}



The \emph{support} of $\sigma \in S_n$ is $\supp(\sigma) = \{i \in [n] : \sigma(i) \neq i\}$.
\begin{lem} \label{lem:hypertree-cycle-type}
If $T$ is a hypertree on $[n]$ and $\Sigma \subseteq S_n$ is a collection of cycles whose supports are exactly the edges of $T$, then $\prod_{e \in \Sigma} \sigma$ is an $n$-cycle, regardless of the order in which the product is taken. \end{lem}

\begin{proof} Any hypertree has a leaf, as can be seen using the same argument one would use for a tree: starting at an arbitrary vertex and trying to follow a path for as long as possible, we are eventually forced to repeat a vertex or edge that has already been used; in the former case we have found a cycle, and in the latter we have found a leaf.

So, let $v$ be a leaf of $T$. Form a new hypergraph $T'$ by removing $v$ from $V(T)$ and from the edge $\{v_1, \ldots, v_k, v\}$ it was in, and deleting the modified edge if it becomes a singleton. Evidently $T'$ is still a hyperforest, and is connected because $v$ was a leaf. Write the product of cycles associated to $T$ as $\sigma (v_1\, \cdots \, v_k\, v) \tau$ where $\sigma$ and $\tau$ are products of cycles. By induction, $\rho = \sigma (v_1\, \cdots \, v_k) \tau$ is an $(n-1)$-cycle with support $[n] \setminus \{v\}$. Now
\begin{align*}
    \sigma (v_1\, \cdots \, v_k\, v) \tau &= \rho \tau^{-1} (v_k\,\, v) \tau = \rho (\tau^{-1}(v_k)\,\, \tau^{-1}(v))\\
    & = \rho(\tau^{-1}(v_k)\,\, v) \quad \text{(since $v \notin \supp(\tau)$)}
\end{align*}
Since $\rho$ is an $(n-1)$-cycle with support containing $\tau^{-1}(v_k)$ but not $v$, the product $\rho(\tau^{-1}(v_k)\,\, v)$ is an $n$-cycle as desired.
\end{proof}

\begin{cor} \label{cor:hyperforest-cycle-type} If $F$ is a hyperforest and $\Sigma$ is a collection of cycles with supports $E(F)$, then $\cyc\left(\prod_{\sigma \in \Sigma} \sigma\right) = \type(F)$. \end{cor}

The power sum expansion of $X_G$ for hypergraphs can be described in exactly the same way as for graphs.
\begin{thm}[\cite{stanley-chromatic-2}, Theorem 3.5] \label{thm:power-sum-expansion-hypergraphs}
$\displaystyle X_G = \sum_{S \subseteq E(G)} (-1)^{|S|} p_{\type(G_S)}$ for any hypergraph $G$.
\end{thm}

\begin{defn} Given $e \subseteq [n]$, let
\begin{equation*}
\alpha_{e} = 1 - \frac{1}{(|e|-1)!}\sum_{\sigma} \sigma,
\end{equation*}
where $\sigma$ runs over all cycles in $S_n$ with support $e$. Given a hyperforest $F$ on $[n]$ and an ordering $\pi$ of its edges, define the \emph{hyperforest operator}
\begin{equation*}
\alpha_{F, \pi} = n!\prod_{e \in E(F)} \alpha_e,
\end{equation*}
with the product being taken in the order dictated by $\pi$.
\end{defn}

\begin{ex}
If $F$ has vertices $\{1,2,3,4\}$ and edges $\{1,2,3\}$ and $\{3,4\}$, then
\begin{equation*}
\alpha_F = 4!\left(1 - \frac{1}{2}(1\,2\,3) - \frac{1}{2}(1\,3\,2)\right)(1 - (3\,4)).
\end{equation*}
\end{ex}

\begin{lem} \label{lem:hyperforest-operator-ch} $\ch \alpha_{F, \pi} = X_F$ for any edge ordering $\pi$ of a hyperforest $F$. \end{lem}

\begin{proof}
$\alpha_{F,\pi}$ is the sum over all subsets $S = \{e_1 < \cdots < e_p\} \subseteq E(F)$ of the expressions
\begin{equation} \label{eq:hyperforest-operator-ch}
(-1)^{|S|} n!\sum_{\sigma_1, \ldots, \sigma_p} \sigma_1 \cdots \sigma_p \prod_{e \in S} \frac{1}{(|e|-1)!}
\end{equation}
where $\sigma_1, \ldots, \sigma_p$ range over the cycles with supports $e_1, \ldots, e_p$. By Corollary~\ref{cor:hyperforest-cycle-type}, the image under $\ch$ of \eqref{eq:hyperforest-operator-ch} is
\begin{equation*}
(-1)^{|S|}  \sum_{\sigma_1, \ldots, \sigma_p} p_{\type(F_S)} \prod_{e \in S} \frac{1}{(|e|-1)!} = (-1)^{|S|} p_{\type(F_S)}.
\end{equation*}
The lemma now follows by Theorem~\ref{thm:power-sum-expansion-hypergraphs}.
\end{proof}

In what follows, we consider $\CC[S_n]$ as an inner product space by taking the elements of $S_n$ to form an orthonormal basis, writing $(\cdot, \cdot)$ for the inner product. Given $\alpha = \sum_{\sigma} c_{\sigma} \sigma \in \CC[S_n]$, let $\alpha^* = \sum_{\sigma} \overline{c}_{\sigma} \sigma^{-1}$.
\begin{prop} \label{prop:hermitian} Viewed as the operator $x \mapsto \alpha x$, the adjoint of $\alpha \in \CC[S_n]$ is $\alpha^*$. \end{prop}
\begin{proof}
It suffices to check that the adjoint of $\sigma \in S_n$ is $\sigma^{-1}$, which follows from the computation
\begin{equation*}
(\sigma\rho, \tau) = \delta_{\sigma\rho , \tau} = \delta_{\rho, \sigma^{-1}\tau } = (\rho, \sigma^{-1}\tau )
\end{equation*}
for all $\rho, \tau \in S_n$.
\end{proof}

\begin{lem} \label{lem:positive-semidefinite-edge-operator} For any $e \subseteq [n]$, the operator $\alpha_e$ is positive semidefinite. \end{lem}
\begin{proof} Using, say, the Murnaghan-Nakayama rule, one can obtain the explicit expansions
\begin{equation*}
p_m = \sum_{k=1}^m (-1)^k s_{(k, 1^{m-k})} \qquad \text{and} \qquad p_1^m = \sum_{\lambda \vdash m} f^{\lambda} s_{\lambda},
\end{equation*}
where $f^{\lambda} = \chi^\lambda(1)$ is the number of standard Young tableaux of shape $\lambda$. These expansions make it clear that $|e|!\ch(\alpha_e) = p_1^{|e|} - p_{|e|}$ is Schur positive, and since $\alpha \in Z(\CC[S_n])$, this implies $\alpha_e$ has nonnegative eigenvalues. Proposition~\ref{prop:hermitian} shows that $\alpha_e$ is Hermitian, so it is positive semidefinite.
\end{proof}

\begin{rem} The proof of Lemma~\ref{lem:hyperforest-operator-ch} actually shows that $\ch(\alpha_{F, \pi}) = X_F$ holds for \emph{any} choice of $\alpha_e = 1 - \sum_{\sigma} c_{\sigma} \sigma$ where $\sigma$ ranges over cycles with support $e$ and $\sum_{\sigma} c_{\sigma} = 1$. For our applications we want $\alpha_e$ to be positive semidefinite, and we made the choice of $\alpha_e$ we did because it allows for the proof of Lemma~\ref{lem:positive-semidefinite-edge-operator} given above, which seems natural in our setting. However, it is worth noting that a more direct computation will show that $1 - \frac{1}{2}\sigma - \frac{1}{2}\sigma^{-1}$ is also positive semidefinite for any cycle $\sigma$.
\end{rem}

We can now prove one of the main theorems of this section.
\begin{thm} \label{thm:edge-bipartite-hypergraphs} Any edge 2-colorable hyperforest $G$ is Schur positive. \end{thm}
\begin{proof} Choose a proper $2$-coloring of the edges of $G$ by colors ``even'' and ``odd''. Define
\begin{equation*}
\alpha_{\text{odd}} = \prod_{\substack{e \in E(G) \\ \text{$e$ odd}}} \alpha_{e} \qquad \text{and} \qquad \alpha_{\text{even}} = \prod_{\substack{e \in E(G) \\ \text{$e$ even}}} \alpha_e.
\end{equation*}
By Lemma~\ref{lem:hyperforest-operator-ch}, $\ch(\alpha_{\text{odd}}\alpha_{\text{even}}) = X_G$. The factors of $\alpha_{\text{odd}}$ commute because the odd edges $e$ are pairwise disjoint. Since each factor of $\alpha_{\text{odd}}$ is positive semidefinite by Lemma~\ref{lem:positive-semidefinite-edge-operator}, so is $\alpha_{\text{odd}}$ itself---as of course is $\alpha_{\text{even}}$. But the product of two positive semidefinite operators has nonnegative eigenvalues \cite[Corollary 7.6.2]{horn-johnson}, so $X_G$ is Schur positive by Lemma~\ref{lem:schur-positive}.
\end{proof}

Suppose $f(x)$ is a formal power series over a field with $f(0) = 0$ and $f'(0) \neq 0$, so that the compositional inverse $f^{-1}(x)$ exists as a formal power series. The formal power series $f(f^{-1}(x_1) + f^{-1}(x_2) + \cdots)$ is called the \emph{formal group law} associated to $f(x)$. Evidently $f(f^{-1}(x_1) + f^{-1}(x_2) + \cdots)$ is symmetric in $x_1, x_2, \ldots$, although it may lie in the completion of $\Lambda$ rather than $\Lambda$ itself.

Taylor \cite{jair-thesis} showed that if $f(x)$ is the exponential or ordinary generating function of an appropriate family of combinatorial objects---more precisely, of a species equipped with an operation satisfying certain axioms---then $f(f^{-1}(x_1) + f^{-1}(x_2) + \cdots)$ is an explicit positive linear combination of chromatic symmetric functions $X_H$ for some hypergraphs $H$. These hypergraphs often turn out to be linear interval hypergraphs, so that the associated formal group law is Schur positive by Theorem~\ref{thm:edge-bipartite-hypergraphs}. The next theorem records some such cases.

\begin{thm} \label{thm:formal-group-laws} \upshape If $a_n$ counts any of the following, the formal group law associated to the ordinary generating function $\sum_{n=1}^{\infty} a_n x^n$ is Schur positive:

\begin{enumerate}[(a)]
\item Permutations of $[n]$.
\item Plane trees where no node has exactly one child, with leaves labeled by $[n]$ from left to right.
\item \emph{$L$-admissible lattice paths} with $n$ steps: here $L \subseteq \ZZ$ is a finite set, and an $L$-admissible lattice path is a sequence $s_1, \ldots, s_n$ where $s_1 = s_n = 0$ and $s_{i+1} - s_i \in L$ for each $i$.
\end{enumerate}

Likewise, if $b_n$ counts any of the following, the formal group law associated to the exponential generating function $\sum_{n=1} \frac{b_n}{n!} x^n$ is Schur positive:
\begin{enumerate}[(a$'$)]
\item Permutations of $[n]$.
\item Rooted trees in which no node has exactly one child, with leaves labeled by $[n]$.
\item Posets with a minimum and a maximum whose elements are labeled by $[n]$.
\end{enumerate}

The same Schur positivity statement also holds for the \emph{Hadamard product} of any two of these exponential generating functions, the Hadamard product of $\sum_{n=0}^{\infty} a_n \frac{x^n}{n!}$ and $\sum_{n=0}^{\infty} b_n \frac{x^n}{n!}$ being $\sum_{n=0}^{\infty} a_n b_n \frac{x^n}{n!}$.
\end{thm}

\begin{proof}
    If $f$ is the ordinary generating function of what is called a \emph{contractible $\mathbb{L}$-species} in \cite{jair-thesis}, then \cite[Theorem 6.1]{jair-thesis} shows how to write $f^{-1}(f(x_1) + f(x_2) + \cdots)$ in the form $\sum_{H} X_{H}$, where $H$ runs over a certain set of hypergraphs. The ordinary generating functions (a), (b), and (c) all correspond to contractible $\mathbb{L}$-species, and as noted in \cite[\S 9.2]{jair-thesis}, all of the associated hypergraphs $H$ are linear interval hypergraphs, so the theorem in these cases follows from Theorem~\ref{thm:edge-bipartite-hypergraphs}.

    If $f$ is the exponential generating function of a \emph{contractible species}, then there is again an expression $f^{-1}(f(x_1) + f(x_2) + \cdots) = \sum_{H} \frac{1}{|V(H)|!} X_H$. The $H$ appearing in the sum are not always edge 2-colorable hypertrees, but in the cases (a$'$), (b$'$), and (c$'$) they are linear interval hypergraphs. Since this fact is not explicitly mentioned in \cite{jair-thesis}, we verify it here, although nothing in the rest of the paper relies on the present theorem. We use the descriptions of the relevant hypergraphs $H$ from, respectively, \S 3.4, \S 3.2, and \S 3.5 of \cite{jair-thesis}.

    \begin{enumerate}[(a$'$)]
        \item $H$ runs over the set of path graphs \cite[\S 3.4]{jair-thesis}.
        \item $H$ is always a hypergraph with disjoint edges \cite[\S 3.2]{jair-thesis}.
        \item Let $P$ be a finite poset with minimum and maximum elements. Define $I'(P)$ to be the set of non-singleton intervals $U = [a,b] \subseteq P$ such that if $p \in P \setminus U$ and $u \in U$, then $p \leq u$ if and only if $p \leq a$, and $p \geq u$ if and only if $p \geq b$. Let $I(P)$ be the set of elements of $I'(P)$ which are minimal under containment. Then $H$ can be taken to run over a set of hypergraphs whose edge sets have the form $I(P)$ \cite[\S 3.5]{jair-thesis}.

        Suppose $H$ is such a hypergraph. Let us first see that $|e \cap e'| \leq 1$ for distinct $e, e' \in E(H)$. Suppose $[a,b], [c,d] \in I(P)$ are distinct and that $x \in [a,b] \cap [c,d]$. By minimality we cannot have $[a,b] \subseteq [c,d]$, so either $a \notin [c,d]$ or $b \notin [c,d]$. Similarly, either $c \notin [a,b]$ or $d \notin [a,b]$. It suffices to consider the following two cases.
        \begin{itemize}
            \item Assume $a \notin [c,d]$ and $c \notin [a,b]$. Then $a \leq x \in [c,d]$ implies $a \leq c$, and $c \leq x \in [a,b]$ implies $c \leq a$. Thus $a = c$, but this contradicts $a \notin [c,d]$.

            \item Assume $a \notin [c,d]$ and $d \notin [a,b]$. Then $a \leq c$ as before, while now $d \geq x \in [a,b]$ implies $d \geq b$. If $b \notin [c,d]$, then $b \geq x \in [c,d]$ implies $b \geq d$, a contradiction. Thus $a \leq c \leq b \leq d$, so $[a,b] \cap [c,d] = [c,b]$. One checks that $[c,b]$ has the property required for membership in $I'(P)$ so long as it is not a singleton, but this would contradict minimality of $[a,b]$. Therefore $c = b$, and $[a,b] \cap [c,d] = \{b\}$.
        \end{itemize}

        A path in $L(H)$ is a sequence of edges $[a_1, b_1], \ldots, [a_k, b_k]$ of $H$ such that $[a_i,b_i] \cap [a_{i-1}, b_{i-1}] \neq \emptyset$ for each $i$. The previous paragraph shows that $a_i = b_{i-1}$ for each $i$, so $a_1 \leq \cdots \leq a_k$. The acyclicity of $\leq$ therefore implies that $L(H)$ is a forest. If $[a,b], [b,c], [c,d]$ and $[a,b], [b,c], [c,e]$ are two paths of length 3 in $L(H)$, then $[c,d] \cap [c,e] \neq \emptyset$ implies $d = e$, so $L(H)$ has no vertices of degree more than two. Thus, $H$ is a linear interval hypergraph.
    \end{enumerate}

    Finally, suppose $f(x) = \sum_{n=0}^{\infty} a_n b_n \frac{x^n}{n!}$ is the Hadamard product of two of the exponential generating functions described above. As described in \cite[\S 5]{jair-thesis}, there is an expression $f^{-1}(f(x_1) + f(x_2) + \cdots) = \sum_H \frac{1}{|V(H)|!} X_H$ where $H$ runs over hypergraphs with edge sets of the form $E(H_1) \cap E(H_2)$, with $H_1$ and $H_2$ being two hypergraphs of the types described in (a$'$), (b$'$), or (c$'$). We saw that such $H_1$ and $H_2$ are linear interval hypergraphs, which implies $H$ is.

\end{proof}

\section{Pointed chromatic symmetric functions}
\label{sec:pointed-chrom-sym}

A \emph{pointed graph} is a pair $(G,v)$ with $v \in V(G)$. Write $\type_{v}^-(G)$ for the partition whose parts are the sizes of the connected components of $G$ not containing $v$, and $\type_{v}^+(G)$ for the size of the component containing $v$. Recall that if $S \subseteq E(G)$, then $G_S$ denotes the graph with vertices $V(G)$ and edges $S$.
\begin{defn} \label{defn:pointed-chrom-sym} The \emph{pointed chromatic symmetric function} of a pointed graph $(G,v)$ is
\begin{equation*}
X_{G,v} = \sum_{S \subseteq E(G)} (-1)^{|S|} p_{\type_v^-(G_S)} t^{\type^+_v(G_S)-1}  \in \Lambda[t].
\end{equation*}
\end{defn}

\begin{ex}
If $G = P_2 = \raisebox{-1mm}{\begin{tikzpicture}[scale=0.2, every node/.style={circle, black, fill, inner sep=0pt, minimum size=2.5pt}]  \node {} child { node {} }; \end{tikzpicture}}$\,, then
    \begin{alignat*}{2}
        S = \emptyset& \leadsto \type^{-}_1(G_S) = (1), \quad &\type^{+}_1(G_S) = 1\\
        S = \{(1,2)\}& \leadsto \type^{-}_1(G_S) = \emptyset, \quad &\type^{+}_1(G_2) = 2
    \end{alignat*}
    so $X_{P_2, 1} = p_1 - t$.
    \begin{center}
    
    \end{center}
\end{ex}

\begin{ex}
    If $G$ is a disjoint union $G_1 \sqcup G_2$ and $v \in G_1$, then $X_{G,v} = X_{G_1,v}X_{G_2}$. In particular, if $v$ is an isolated vertex of $G$ then $X_{G,v} = X_G$.
\end{ex}

The \emph{wedge sum} $(G,v) \vee (H,w)$ is the pointed graph obtained from $G \sqcup H$ by identifying $v$ and $w$, with $v = w$ as the distinguished vertex. The next proposition is analogous to the fact that $X_{G \sqcup H} = X_G X_H$.
\begin{prop} $X_{G \vee H, v} = X_{G, v} X_{H, w}$ for any pointed graphs $(G,v)$ and $(H,w)$. \end{prop}
    \begin{proof} For subsets $S \subseteq E(G)$ and $S' \subseteq E(H)$, we have $(G \vee H)_{S \cup S'} = G_S \vee H_{S'}$, which implies
        \begin{align*}
            &\type_v^+((G \vee H)_{S \cup S'}) = \type_v^+(G_S) + \type_v^+(H_{S'}) - 1\\
            &\type_v^-((G \vee H)_{S \cup S'}) = \type_v^-(G_S) \cup \type_v^-(H_{S'}),
        \end{align*}
        where $\lambda \cup \mu$ is the partition whose parts are the parts of $\lambda$ together with the parts of $\mu$.
    \end{proof}

Pointed chromatic symmetric functions satisfy a deletion-contraction recurrence. Given $e \in E(G)$, let $G \setminus e$ be the graph obtained from $G$ by deleting the edge $e$, and $G/e$ the graph obtained by contracting it.
\begin{lem} \label{lem:del-con} If $e$ is incident to $v \in V(G)$, then $X_{G,v} = X_{G\setminus e, v} - tX_{G/e, v}$. \end{lem}
    \begin{proof} In Definition~\ref{defn:pointed-chrom-sym}, sum separately over those sets $S \subseteq E(G)$ in which $e$ does not appear and those in which it does. These two sums are, respectively, $X_{G\setminus e, v}$ and $-tX_{G/e, v}$. To be precise, the second sum is $-tX_{G/e, v}$ where $G/e$ is considered as a \emph{multigraph}---for instance, if $e = (v,w)$ and $G$ contains a triangle $(v,w),(u,v),(u,w) \in E(G)$, then the edge $(u,v)$ appears twice in $G/e$. However, one can check by inclusion-exclusion that if $H$ is a multigraph and $H'$ is its underlying graph, then $X_{H,v} = X_{H',v}$.
    \end{proof}
\begin{cor} \label{cor:monomial-positive} After applying the $\Lambda$-linear map $\Lambda[t] \to \Lambda[t]$ sending $t \mapsto -t$, $X_{G,v}$ becomes nonnegative in the monomial basis $\{m_{\lambda} t^i\}$ of $\Lambda[t]$. \end{cor}
    \begin{proof} Induct using the deletion-contraction recurrence. The base case is when $v$ is an isolated vertex, in which case $X_{G,v} = X_G \in \Lambda$ is monomial-positive.
    \end{proof}
Given Corollary~\ref{cor:monomial-positive}, it would be interesting to find a combinatorial description of the monomial expansion of $X_{G,v}$.

Our goal in the rest of this section is to prove that if $X_{G,v}$ and $X_{H,w}$ are positive in the \emph{pointed Schur basis}, then $X_{G \wedge H, v}$ is Schur positive. Defining this basis and proving the necessary lemmas requires a small detour into representation theory. For $v \in [n]$, let $Z_{n,v}$ be the centralizer of $\CC[S_{n-1}]$ in $\CC[S_n]$, where we identify $S_{n-1}$ with the subgroup of $S_n$ fixing $v$.

\begin{lem}[\cite{vershik-okounkov}, Theorem 2.1] \label{lem:nice-centralizer} Every $x \in \RR[S_n] \cap Z_{n,v}$ is Hermitian, and $Z_{n,v}$ is a semisimple commutative algebra. \end{lem}

Being semisimple, Wedderburn's theorem says that the algebra $Z_{n,v}$ has a unique basis of orthogonal idempotents which add to $1$. These idempotents can be described as follows. Given a part $i$ of $\lambda \vdash n$, let $\lambda^{\downarrow i}$ be $\lambda$ with one copy of $i$ replaced by $i-1$ (or deleted if $i = 1$). Let $\epsilon_{\lambda} = \frac{\chi^{\lambda}(1)}{n!}\sum_{\sigma \in S_n} \chi^{\lambda}(\sigma)\sigma \in S_n$ be the idempotent in $Z(\CC[S_n])$ associated to the character $\chi^{\lambda}$. Define $\epsilon_{\lambda,i} \eqdef \epsilon_{\lambda^{\downarrow i}} \epsilon_{\lambda} $, viewing $\epsilon_{\lambda^{\downarrow i}} \in \CC[S_{n-1}]$ as an element of $\CC[S_n]$. The $\epsilon_{\lambda,i}$ lie in $Z_{n,v}$ and are again idempotents, and one can show that $\sum_{i \in \lambda} \epsilon_{\lambda,i} = \epsilon_{\lambda}$ (this is basically the branching rule for restricting representations from $S_n$ to $S_{n-1}$), which implies that they are the basis of primitive orthogonal idempotents of $Z_{n,v}$. The next two definitions are due to Strahov \cite{strahov}.

\begin{defn} For $v \in [n]$, let $\cyc_v^+(\sigma)$ be the size of the cycle of $\sigma$ containing $v$, and $\cyc_v^-(\sigma)$ the partition of cycle sizes of the cycles not containing $v$. The \emph{pointed Frobenius characteristic} $\ch_v : \CC[S_n] \to \Lambda[t]$ is the linear map sending $\sigma \in S_n$ to
    \begin{equation*}
        \frac{1}{(n-1)!} p_{\cyc_v^-(\sigma)} t^{\cyc_v^+(\sigma)-1}.
    \end{equation*}
\end{defn}

\begin{defn} The \emph{pointed Schur function} associated to $\lambda \vdash n$ and $i \in \lambda$ is
\begin{equation*}
s_{\lambda,i} = \frac{n!}{\chi^{\lambda}(1)} \ch_n(\epsilon_{\lambda,i}).
\end{equation*} \end{defn}
Pointed Schur functions form a basis of $\Lambda[t]$ because the idempotents $\epsilon_{\lambda,i}$ form a basis of $Z_{n,v}$. Note that $s_{\lambda,i}$ has degree $|\lambda|-1$.

\begin{exlem} \label{exlem:pointed-e} $s_{(1^n),1} = \sum_{k=1}^{n} (-t)^{k-1} e_{n-k}$.     \end{exlem}
    
    \begin{proof}
        Since $\chi^{(1^n)}(\sigma) = \sgn(\sigma)$ we have
        \begin{align*}
            \epsilon_{(1^n),1} = \epsilon_{(1^{n-1})} \epsilon_{(1^n)} &= \frac{1}{n!(n-1)!}\sum_{\sigma \in S_{n-1}} \sgn(\sigma) \sigma \sum_{\tau \in S_{n}} \sgn(\tau)\tau\\
            &= \frac{1}{n!(n-1)!}\sum_{\sigma \in S_{n-1}} \sgn(\sigma) \sgn(\sigma) \sum_{\tau \in S_{n}} \sgn(\tau)\tau\\
            &= \frac{1}{n!} \sum_{\tau \in S_{n}} \sgn(\tau)\tau.
        \end{align*}
        Thus,
        \begin{equation}
            s_{(1^n),1} = \frac{n!}{\chi^{(1^n)}(1)} \ch_n(\epsilon_{(1^{n-1})} \epsilon_{(1^n)} ) = \ch_n\left( \sum_{\tau \in S_{n}} \sgn(\tau)\tau \right). \label{eq:enn}
        \end{equation}
    
        Letting $z_{\lambda}$ be the size of the centralizer of a permutation of cycle type $\lambda$ (so there are $n!/z_{\lambda}$ permutations in $S_n$ of cycle type $\lambda$), one verifies similarly that the number of permutations $\sigma \in S_n$ with $\cyc_n^-(\sigma) = \lambda$ is $(n-1)!/z_{\lambda}$. Applying this fact, and breaking up the sum \eqref{eq:enn} according to the size $k$ of the cycle containing $n$,
        \begin{align*}
            &s_{(1^n),1} = \sum_{k=1}^n (-1)^{k-1}  \sum_{\lambda \vdash n-k} \sgn(\lambda) \frac{p_{\lambda}}{z_{\lambda}} t^{k-1} .
        \end{align*}
        Here $\sgn(\lambda) = (-1)^{n - \ell(\lambda)}$ is the sign of a permutation of cycle type $\lambda$. The lemma now follows from the symmetric function identity $\sum_{\lambda \vdash m} \sgn(\lambda) \frac{p_{\lambda}}{z_{\lambda}} = e_{m}$ \cite[Proposition 7.7.6]{stanley-ec2}.
    
    \end{proof}

Let $S_X$ be the group of permutations of a set $X$, and let $[a,b] \eqdef \{a, a+1, \ldots, b\}$. Define a bilinear product $\circ : \CC[S_n] \times \CC[S_{[n,n+m-1]}] \to Z_{n+m-1,n}$ by
\begin{equation*}
    (\alpha, \beta) \mapsto \alpha \circ \beta \eqdef \frac{1}{(n-1)!(m-1)!}\sum_{\substack{\sigma \in S_{n+m-1}\\ \sigma(n) = n}} \sigma \alpha \beta \sigma^{-1}.
\end{equation*}

\begin{prop} \label{prop:pointed-frob-multiplicative} $\ch_n(\alpha \circ \beta) = \ch_n(\alpha)\ch_n(\beta)$. \end{prop}
    \begin{proof}
        It suffices to assume $\alpha \in S_n$ and $\beta \in S_{[n,n+m-1]}$, in which case the proposition follows from the identities
        \begin{align*}
            &\cyc_n^-(\alpha\beta) = \cyc_n^-(\alpha) \cup \cyc_n^+(\beta)\\
            &\cyc_n^+(\alpha\beta) = \cyc_n^+(\alpha) + \cyc_n^+(\beta) - 1.
        \end{align*}
    \end{proof}

Let $\psi : \Lambda[t] \to \Lambda$ be the linear map with $\psi(p_{\lambda} t^{i-1}) = p_{\lambda \cup i}$ for $i \geq 1$. It can be shown that $\psi(s_{\lambda,i}) = (n{-}1)^{-1} f^{\lambda^{\downarrow i}} s_{\lambda}$, so if $f$ is pointed Schur positive then $\psi(f)$ is Schur positive. The products $s_{\lambda,i}s_{\mu,j}$ are usually \emph{not} pointed Schur positive, but a weaker result holds.
\begin{lem} \label{lem:pointed-schur-product} $\psi(s_{\lambda,i}s_{\mu,j})$ is Schur positive. \end{lem}
    \begin{proof}
    Using Proposition~\ref{prop:pointed-frob-multiplicative} and the fact that $\ch = \psi \circ \ch_v$,
    \begin{equation*}
        \psi(s_{\lambda,i}s_{\mu,j}) = \frac{|\lambda|!}{\chi^{\lambda}(1)} \frac{|\mu|!}{\chi^{\mu}(1)}\ch(\epsilon_{\lambda,i} \circ \epsilon_{\mu,j}).
    \end{equation*}
    By Lemma~\ref{lem:trace}, it is enough to show that $\tr((\epsilon_{\lambda,i} \circ \epsilon_{\mu,j})\epsilon_{\nu})$ is nonnegative for all $\nu \vdash |\lambda|+|\mu|-1$. Now,
    \begin{align} \label{eq:trace}
        (n-1)!(m-1)!\tr((\epsilon_{\lambda,i} \circ \epsilon_{\mu,j})\epsilon_{\nu}) &= \sum_{\substack{\sigma \in S_{n+m-1}\\ \sigma(n)=n}} \tr(\sigma \epsilon_{\lambda,i}\epsilon_{\mu,j} \sigma^{-1} \epsilon_{\nu}) \nonumber\\
        &= \sum_{\substack{\sigma}} \tr(\epsilon_{\lambda,i}\epsilon_{\mu,j} \sigma^{-1} \epsilon_{\nu} \sigma) \nonumber \\
        &= \sum_{\substack{\sigma}} \tr(\epsilon_{\lambda,i}\epsilon_{\mu,j} \epsilon_{\nu}) \qquad \text{(as $\epsilon_{\nu}$ is central)}.
    \end{align}

    Since the characters of $S_n$ are real-valued, Lemma~\ref{lem:nice-centralizer} shows that $\epsilon_{\lambda,i}$ and $\epsilon_{\mu,j}$ (and of course $\epsilon_{\nu}$) are Hermitian. Being idempotent, they are therefore positive semidefinite, as is $\epsilon_{\mu,j}\epsilon_{\nu}$ since $\epsilon_{\nu}$ is central. It follows that $\epsilon_{\lambda,i}\epsilon_{\mu,j} \epsilon_{\nu}$ has nonnegative eigenvalues and that the sum \eqref{eq:trace} is nonnegative.
    \end{proof} 

    We now come to the main result of this section.
\begin{thm} Let $(G,v)$ and $(H,w)$ be pointed graphs. If $X_{G,v}$ and $X_{H,w}$ are pointed Schur positive, then $X_{G \vee H}$ is Schur positive. \end{thm} 

    \begin{proof} It is clear from the definitions that $\psi(X_{G \vee H, v}) = X_{G \vee H}$. Since $X_{G \vee H, v}= X_{G,v}X_{H,w}$ by Proposition~\ref{prop:pointed-frob-multiplicative}, the theorem follows from Lemma~\ref{lem:pointed-schur-product}.
    \end{proof}
In Section~\ref{sec:pointed-chrom-sym-path} we will obtain explicit expressions for pointed chromatic symmetric functions of paths (when the distinguished vertex is a leaf) and cycles, which in particular will be pointed Schur positive. Even without such expressions we can obtain a stronger result for paths than pointed Schur positivity. First, a lemma analogous to Lemma~\ref{lem:trace}.

\begin{lem} \label{lem:pointed-trace} For any $\alpha \in \CC[S_n]$, the coefficient of $s_{\lambda,i}$ in $\ch_v(\alpha)$ is $\frac{\chi^{\lambda}(1)}{n!\rank(\epsilon_{\lambda,i})}\tr(\alpha \epsilon_{\lambda,i})$. \end{lem} 
\begin{proof}
First suppose $\alpha \in Z_{n,v}$. Since the $\epsilon_{\lambda,i}$ are commuting orthogonal idempotents spanning $Z_{n,v}$, the coefficient of $\epsilon_{\lambda,i}$ in $\alpha$ is $\tr(\alpha \epsilon_{\lambda,i})/\rank(\epsilon_{\lambda,i})$. By definition of $s_{\lambda,i}$, this coefficient is equal to $\frac{n!}{\chi^{\lambda}(1)} [s_{\lambda,i}]\ch_v(\alpha)$, so the lemma holds when $\alpha \in Z_{n,v}$.

Given an arbitrary $\alpha \in \CC[S_n]$, define $\beta = \frac{1}{(n-1)!}\sum_{\sigma \in S_{n-1}} \sigma \alpha \sigma^{-1}$. Then $\beta \in Z_{n,v}$ and $\ch_v(\beta) = \ch_v(\alpha)$. On the other hand,
\begin{align*}
    \tr(\beta \epsilon_{\lambda,i}) &= \frac{1}{(n-1)!}\sum_{\sigma \in S_{n-1}} \tr(\sigma \alpha \sigma^{-1} \epsilon_{\lambda,i}) = \frac{1}{(n-1)!}\sum_{\sigma \in S_{n-1}} \tr( \alpha \sigma^{-1} \epsilon_{\lambda,i}\sigma)\\
    &= \frac{1}{(n-1)!}\sum_{\sigma \in S_{n-1}} \tr(\alpha \epsilon_{\lambda,i}) \qquad \text{(since $\epsilon_{\lambda,i} \in Z_{n,v}$)}\\
    &= \tr(\alpha \epsilon_{\lambda,i}).
\end{align*}
Therefore the lemma holds by applying the previous paragraph to $\beta$.
\end{proof}

\begin{thm} \label{thm:path-pointed-schur-positive} Let $G$ be a graph on $[n,n+m-1]$ and $P$ the path on $[n]$, both with distinguished vertex $n$. If $X_{G,n}$ is pointed Schur positive, then $X_{G \vee P, 1}$ is also pointed Schur positive \end{thm}

\begin{proof}
Let $\beta \in Z_{[n,n+m-1],n}$ be such that $\ch_n(\beta) = X_{G,n}$. Define
\begin{equation*}
    A = \prod_{\substack{i=1\\ \text{$i$ odd}}}^{n-1} (1 - (i\,\,i{+}1)) \quad \text{and} \quad B = \prod_{\substack{i=1\\ \text{$i$ even}}}^{n-1} (1 - (i\,\,i{+}1)).
    \end{equation*}
Observe that $\frac{(n+m-1)!}{n!}\ch_1(A \beta B) = X_{G \vee P, 1}$ (cf. Lemma~\ref{lem:tree-cycle-type}). By Lemma~\ref{lem:pointed-trace}, it suffices to see that $\tr(\epsilon_{\lambda,i} A\beta B) \geq 0$ for each idempotent $\epsilon_{\lambda,i} \in Z_{n+m-1,1}$.

Both $A$ and $B$ are positive semidefinite, being the product of commuting positive semidefinite operators. Both $\epsilon_{\lambda,i}$ and $\beta$ are also positive semidefinite: they are Hermitian by Lemma~\ref{lem:nice-centralizer}, $\epsilon_{\lambda,i}$ is an idempotent, and $\beta$ has nonnegative eigenvalues because $\ch_n(\beta) = X_{G,n}$ is pointed Schur positive (applying Lemma~\ref{lem:pointed-trace}). Now consider two cases:
\begin{itemize}
    \item $n$ odd: Here $A$ is in $\CC[S_{n-1}]$, so commutes with $\beta \in \CC[S_{[n,n+m-1]}]$. Similarly, $B \in \CC[S_{[2,n]}]$ while $\epsilon_{\lambda,i}$ is in $Z_{n+m-1,1}$, the centralizer of $S_{[2,n+m-1]}$. Therefore both $A\beta$ and $B\epsilon_{\lambda,i}$ are positive semidefinite, so $(A\beta)(B\epsilon_{\lambda,i})$ has nonnegative trace.

    \item $n$ even: Here $\beta$ and $B$ are both in $\CC[S_{[2,n]}]$, so commute with $\epsilon_{\lambda,i}$. In fact, $B$ is in $\CC[S_{[2,n-1]}]$, so it also commutes with $\beta$. Therefore $\beta B \epsilon_{\lambda,i}$ is positive semidefinite, so $A(\beta B \epsilon_{\lambda,i})$ has nonnegative trace.
\end{itemize}
\end{proof}

Taking $G$ to be a single vertex shows that $X_{P_n,n}$ is pointed Schur positive. As another example, $X_{C,v}$ is pointed Schur positive where $C$ is a cycle graph (Theorem~\ref{thm:cycle-pointed-chrom}), so any graph
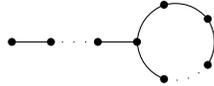
\begin{center}
    \begin{tikzpicture}[scale=0.52, every node/.style = {circle, fill, inner sep=0, minimum size=3pt}]
        \draw (-1,0) node {} to[bend left] (-0.309, 0.951) node {} to[bend left] (0.809, 0.588) node {} to[bend left] (0.809, -0.588) node {};
        \draw[loosely dotted]  (0.809, -0.588) to[bend left] (-0.309, -0.951) node {};
        \draw (-0.309, -0.951) to[bend left] (-1,0);
        \draw (-1,0) to (-2,0) node {};
        \draw[loosely dotted] (-2,0) to (-3.2,0);
        \draw (-3.2,0) node {} to (-4.2,0) node {};
    \end{tikzpicture}
\end{center}
is pointed Schur positive at the leaf, and any graph
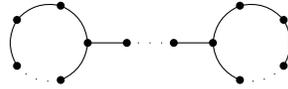
\begin{center}
    \begin{tikzpicture}[scale=0.52, every node/.style = {circle, fill, inner sep=0, minimum size=3pt}]
        \draw (-1,0) node {} to[bend left] (-0.309, 0.951) node {} to[bend left] (0.809, 0.588) node {} to[bend left] (0.809, -0.588) node {};
        \draw[loosely dotted]  (0.809, -0.588) to[bend left] (-0.309, -0.951) node {};
        \draw (-0.309, -0.951) to[bend left] (-1,0);
        \draw (-1,0) to (-2,0) node {};
        \draw[loosely dotted] (-2,0) to (-3.2,0);
        \draw (-3.2,0) node {} to (-4.2,0) node {};
        \draw (-4.2,0) to[bend right] (-4.891, 0.951) node {} to[bend right] (-6.009, 0.588) node {} to[bend right] (-6.009, -0.588) node {};
        \draw[loosely dotted] (-6.009, -0.588) to[bend right] (-4.891, -0.951) node {};
        \draw (-4.891, -0.951) to[bend right] (-4.2,0);
    \end{tikzpicture}
\end{center}
is Schur positive.

The commutativity and semisimplicity of $Z_{n,v}$ arise from the fact that the branching rule for restricting irreducible representations from $S_n$ to $S_{n-1}$ is multiplicity free, and is key in Vershik and Okounkov's approach to $S_n$ representation theory \cite{vershik-okounkov}. When restricting to $S_{n-k}$ with $k > 1$, multiplicities can appear, and the centralizer of $\CC[S_{n-k}]$ is no longer semisimple. On the other hand, irreducible restrictions restricted to $S_{n-2} \times S_2$ \emph{are} multiplicity-free; indeed, the coefficient of $S^{\lambda} \times S^{\mu}$ in $\operatorname{Res}^{S_n}_{S_{n-k} \times S_k} S^{\nu}$ is the Littlewood-Richardson coefficient $c^{\nu}_{\lambda\mu}$, which is $0$ or $1$ when $|\mu| \leq 2$. To get an  analogue of a pointed chromatic symmetric function in this setting, we would choose two distinguished vertices $v, w$ and record two partitions (for each subset of edges): the unordered list of sizes of the connected components not containing $v$ or $w$, and the unordered list of the sizes of the $1$ or $2$ components containing $v$ and $w$. However, it is unclear whether the product of two of these symmetric functions would have a nice interpretation in the same way that the product of pointed chromatic symmetric functions relates to wedge sum.

\section{Pointed chromatic symmetric functions of particular graphs}
\label{sec:pointed-chrom-sym-path}

In this section we give formulas for pointed chromatic symmetric functions of complete graphs, cycle graphs, and paths (with the distinguished vertex being a leaf in the last case). These formulas will be positive in a pointed analogue of the basis of elementary symmetric functions $\{e_{\lambda}\}$. Pointed elementary symmetric functions will themselves be pointed Schur positive, so this is a stronger property than pointed Schur positivity.

\begin{defn}
Given a composition $\alpha$ and a part $i \in \alpha$, define the \emph{pointed elementary symmetric function} $e_{\alpha,i}$ by first defining $e_{i,i} \eqdef s_{(1^i),1}$, and then
\begin{equation*}
e_{\alpha,i} \eqdef e_{i,i} \prod_{j \in \alpha \setminus i} e_j.
\end{equation*}
\end{defn}
We will take $e_{0,0}$ to be $0$. By $\alpha \setminus i$ we mean $\alpha$ with, say, the first instance of $i$ deleted---but we only use this notation in cases where the order of the parts of $\alpha \setminus i$ is not relevant. Likewise, we use $\alpha \cup i$ to mean $\alpha$ with a part $i$ added, and again the position in which $i$ is added will never be relevant.

\begin{lem} \label{lem:pointed-schur-times-schur} Let $\lambda, \mu$ be partitions and $j \in \mu$. Then $s_{\lambda} s_{\mu,j}$ is pointed Schur positive. \end{lem}
    \begin{proof}
        Say $\lambda \vdash n{-}1$ and $\mu \vdash m$. Viewing $\epsilon_{\lambda} \in \CC[S_{n-1}]$ as a member of $\CC[S_n]$, we have $\ch_n(\epsilon_{\lambda}) = \ch(\epsilon_{\lambda}) = s_{\lambda}$. By Proposition~\ref{prop:pointed-frob-multiplicative}, $s_\lambda s_{\mu,j} = \ch_n(\epsilon_{\lambda} \circ \epsilon_{\mu,j})$. Up to multiplication by positive scalars, the pointed Schur coefficients of $\ch_n(\epsilon_{\lambda} \circ \epsilon_{\mu,j})$ are the numbers $\tr(\epsilon_{\nu,k}(\epsilon_{\lambda} \circ \epsilon_{\mu,j}))$ for $\nu \vdash n+m-1$, as per Lemma~\ref{lem:pointed-trace}. But we already showed that these traces are nonnegative in the proof of Lemma~\ref{lem:pointed-schur-product}.
    \end{proof}

\begin{cor} \label{lem:pointed-e-schur-positive} $e_{\lambda,i}$ is pointed Schur positive. \end{cor}

    Recall from Example-Lemma \ref{exlem:pointed-e} that $e_{n,n} = \sum_{k=1}^{n} (-t)^{k-1} e_{n-k}$. Equivalently,
    \begin{equation} \label{eq:pointed-e} 
        \sum_{n=1}^{\infty} e_{n,n} z^n = \frac{z}{1+tz} \sum_{i=0}^{\infty} e_n z^n.
    \end{equation}

\begin{prop} \label{prop:e-times-t} $\displaystyle te_{\lambda,i} = e_{\lambda} - e_{\lambda^{\uparrow i}, i+1} = e_{\lambda \cup 1, 1} - e_{\lambda^{\uparrow i}, i+1}$ for any partition $\lambda$ and $i \in \lambda$, where $\lambda^{\uparrow i}$ is the partition obtained from $\lambda$ by replacing one copy of the part $i$ with $i+1$. 
\end{prop}
\begin{proof}
We need to see that $te_{i,i} = e_i - e_{i+1,i+1}$. By \eqref{eq:pointed-e}, this is equivalent to the identity
\begin{equation*}
    t\cdot \frac{z}{1+tz} \sum_{i=0}^{\infty} e_i z^i = \left(1 - \frac{1}{1+tz}\right) \sum_{i=0}^{\infty} e_i z^i.
\end{equation*}
\end{proof}

\begin{lem} \label{lem:psi-e}
    $\psi(e_{\lambda,i}) = i e_{\lambda}$.
\end{lem}

\begin{proof} Since $\psi$ is a map of $\Lambda$-modules, it suffices to check that $\psi(e_{n,n}) = n e_n$. Indeed,
    \begin{equation*}
        \psi(e_{n,n}) = \sum_{k=1}^n (-1)^{k-1} p_k e_{n-k} \quad \text{(by Example-Lemma~\ref{exlem:pointed-e})}
    \end{equation*}
    and Newton's identity says that this equals $n e_n$.
\end{proof}
In particular, if $X_{G,v}$ is pointed $e$-positive, then $X_G = \psi(X_{G,v})$ is $e$-positive.

\subsection{Complete graphs}
Let $K_n$ be the complete graph on $[n]$, and let $K(n,k)$ be $K_n$ with the edges $(1,n), (2,n), \ldots, (k,n)$ removed. Thus, $K(n,0) = K_n$, while $K(n,n-1)$ is $K_{n-1}$ plus the isolated vertex $n$.
\begin{thm} $X_{K_n, n} = (n-1)!e_{n,n}$. \end{thm}
\begin{proof}    
    \label{thm:complete-graph} Letting $e$ be the edge $(k+1,n)$, we have $K(n,k) \setminus e = K(n,k+1)$ and $K(n,k)/e = K_{n-1}$. By the deletion-contraction recurrence (Lemma~\ref{lem:del-con}),
    \begin{align*}
        X_{K_n, n} &= X_{K(n,1),n} - t X_{K_{n-1},n-1} = X_{K(n,2),n} - 2t X_{K_{n-1},n-1} = \cdots \\
        &= X_{K(n,n-1),n} - (n-1)tX_{K_{n-1},n-1}.
    \end{align*}
    We have $X_{K(n,n-1),n} = X_{K_{n-1}}$, which is easily verified to be $(n-1)!e_{n-1}$, and by induction $X_{K_{n-1},n-1} = (n-2)!e_{n-1,n-1}$. Thus
    \begin{equation*}
        X_{K_n,n} = (n-1)!e_{n-1} - (n-1)!te_{n-1,n-1},
    \end{equation*}
    which is $(n-1)!e_{n,n}$ by Proposition~\ref{prop:e-times-t}.
\end{proof}

\subsection{Paths} 
Let
\begin{equation*}
F(z) = \sum_{n=0}^{\infty} X_{P_n} z^n.
\end{equation*}
Stanley \cite[Proposition 5.3]{stanley-chromatic} showed that
\begin{equation*}
F(z) = \frac{\sum_{i=0}^{\infty} e_i z^i}{1 - \sum_{i=1}^{\infty} (i-1) e_i z^i}.
\end{equation*}
Write $m_i(\lambda)$ for the multiplicity of the part $i$ in $\lambda$, and $m(\lambda)$ for the list of multiplicities. The number of permutations of the list $\lambda$ is then the multinomial coefficient ${\ell(\lambda) \choose m(\lambda)}$. 

\begin{thm} \label{thm:path-pointed-chrom} $\sum_{n=0}^{\infty} X_{P_n,n}z^n = \frac{z}{1+tz} F(z)$, and
\begin{equation} \label{eq:path-pointed-e}
X_{P_n,n} = \sum_{\substack{\lambda \\ i \in \lambda}} {\ell(\lambda \setminus i) \choose m(\lambda \setminus i)} \left[\prod_{j \in \lambda \setminus i} j{-}1\right]  e_{\lambda,i}.
\end{equation}
\end{thm}
\begin{proof}
The deletion-contraction recurrence applied to the edge $(n-1,n)$ of $P_n$ gives
\begin{equation} \label{eq:pointed-path-recurrence}
X_{P_n,n} = X_{P_{n-1}} - tX_{P_{n-1},n-1},
\end{equation}
where we set $X_{P_0, 0} \eqdef 0$. Letting $G(z) = \sum_{n=0}^{\infty} X_{P_n,n}z^n$, the recurrence \eqref{eq:pointed-path-recurrence} gives $G(z) = z F(z) - tz G(z)$, i.e.
\begin{equation*}
G(z) = \frac{z}{1 + tz} F(z).
\end{equation*}

Using equation~\eqref{eq:pointed-e}, this formula for $G(z)$ gives
\begin{equation*}
G(z) = \frac{z}{1+tz} \frac{\sum_{i=0}^{\infty} e_i z^i}{1 - \sum_{i=1}^{\infty} (i-1) e_i z^i} = \frac{\sum_{i=1}^{\infty} e_{i,i} z^i}{1 - \sum_{i=1}^{\infty} (i-1) e_i z^i}.
\end{equation*}
From this generating function we see that the coefficient of $e_{\lambda,i}$ in $X_{P_n, n}$ (where $n = |\lambda|$) is 
\begin{equation*}
    X_{P_n,n} = \sum_{\alpha \vDash n} \left[\prod_{j \in \alpha\setminus \alpha_1} j{-}1\right] e_{\alpha,\alpha_1}\
\end{equation*}
where $\alpha$ runs over compositions of $n$. This is equivalent to the formula given in the statement of the theorem.
\end{proof}

It is often the case that $X_{P_n,v}$ is pointed Schur positive when $v$ is not a leaf, but we have no general classification of when this happens. We originally conjectured that $X_{P_n,k}$ is pointed Schur positive if and only if one of $k$ and $n-k+1$ is odd, together with the exceptional cases that $X_{P_4, 2}$ is pointed Schur positive and $X_{P_n, 2}$ is (apparently) not for $n \geq 3$. This holds for $n \leq 17$, but $X_{P_{18},4}$ is not pointed Schur positive.

\subsection{Cycles}
Let $C_n$ be the cycle graph on $[n]$, with edges $(1,2), (2,3),\ldots, (n{-}1,n)$, $(n,1)$. In particular, $C_1$ is a vertex with a loop attached, and $C_2$ is a pair of vertices with two edges between them. Definition~\ref{defn:pointed-chrom-sym} still makes sense when $G$ has loops or multiple edges, and gives $X_{C_1, 1} = 0$ and $X_{C_2, 2} = p_1 - t$. We define $X_{C_0,0} \eqdef 0$.

\begin{thm} \label{thm:cycle-pointed-chrom}
    \begin{equation*}
        \sum_{n=1}^{\infty} X_{C_n,n} z^n  = \frac{1}{1+tz}\left(\frac{z}{1+tz}F(z) - z\right),
    \end{equation*}
    where $F(z)$ is defined as above. Also,
    \begin{equation} \label{eq:cycle-pointed-e}
        X_{C_n,n} = \sum_{\substack{\lambda \\ i \in \lambda}} {\ell(\lambda \setminus i) \choose m(\lambda \setminus i)} \left[\prod_{j \in \lambda} j{-}1\right]  e_{\lambda,i}.
        \end{equation}
\end{thm}
Note that the only difference between expression \eqref{eq:cycle-pointed-e} for $X_{C_n,n}$ and expression \eqref{eq:path-pointed-e} for $X_{P_n,n}$ in Theorem~\ref{thm:path-pointed-chrom} is that $\prod_{j \in \lambda \setminus i} j{-}1$ is replaced by $\prod_{j \in \lambda} j{-}1$. Thus, Theorem~\ref{thm:cycle-pointed-chrom} asserts that the linear map $\Lambda[t] \to \Lambda[t]$ sending $e_{\lambda,i}$ to $(i-1)e_{\lambda,i}$ maps $X_{P_n,n}$ to $X_{C_n,n}$. 

\begin{proof}[Proof of Theorem~\ref{thm:cycle-pointed-chrom}] Let $H(z) = \sum_{n=2}^{\infty} X_{C_n,n} z^n$.
    The deletion-contraction recurrence gives
    \begin{equation} \label{eq:pointed-cycle-recurrence}
        X_{C_n,n} = X_{P_n,n} - tX_{C_{n-1},n-1}
    \end{equation}
    for $n \geq 2$. Thus, $H(z) = \sum_{n=2}^{\infty} X_{P_n,n}z^n - tz H(z)$, and solving for $H(z)$ and applying Theorem~\ref{thm:path-pointed-chrom} gives
    \begin{equation*}
        \frac{1}{1+tz}\left(\frac{z}{1+tz}F(z) - z\right)
    \end{equation*}

By induction on $n$ and Proposition~\ref{prop:e-times-t},
\begin{align*}
    -tX_{C_{n-1}, n-1} &= \sum_{\substack{\lambda \\ i \in \lambda}} {\ell(\lambda \setminus i) \choose m(\lambda \setminus i)} \left[\prod_{j\in \lambda} j{-}1\right] (e_{\lambda^{\uparrow i}, i+1} - e_{\lambda \cup 1,1})\\
    &= \sum_{\substack{\lambda \\ 2 \leq i \in \lambda}} {\ell(\lambda \setminus i) \choose m(\lambda \setminus i)} \left[(i-2) \prod_{j\in \lambda \setminus i} j{-}1\right] e_{\lambda,i} - \sum_{\lambda} \sum_{i \in \lambda} {\ell(\lambda \setminus i) \choose m(\lambda \setminus i)} \left[\prod_{j \in \lambda} j{-}1\right]    e_{\lambda \cup 1,1}\\
    &= \sum_{\substack{\lambda \\ 2 \leq i \in \lambda}} {\ell(\lambda \setminus i) \choose m(\lambda \setminus i)} \left[(i-2) \prod_{j\in \lambda \setminus i} j{-}1\right] e_{\lambda,i} - \sum_{\lambda} {\ell(\lambda) \choose m(\lambda)} \left[\prod_{j \in \lambda} j{-}1\right]    e_{\lambda \cup 1,1},
\end{align*}
where in the last line we have made the simplification
\begin{equation*}
    \sum_{i \in \lambda} {\ell(\lambda \setminus i) \choose m(\lambda \setminus i)} = \sum_{i \in \lambda} \frac{m_i(\lambda)}{\ell(\lambda)} {\ell(\lambda) \choose m(\lambda)} = {\ell(\lambda) \choose m(\lambda)}.
\end{equation*}
But after reindexing,
\begin{equation*}
    \sum_{\lambda} {\ell(\lambda) \choose m(\lambda)} \left[\prod_{j \in \lambda} j{-}1\right]    e_{\lambda \cup 1,1} = \sum_{\lambda : 1 \in \lambda} {\ell(\lambda \setminus 1) \choose m(\lambda \setminus 1)} \left[\prod_{j \in \lambda \setminus 1} j{-}1\right]    e_{\lambda},
\end{equation*}
so actually
\begin{equation*}
    -tX_{C_{n-1},n-1} = \sum_{\substack{\lambda \\ i \in \lambda}} {\ell(\lambda \setminus i) \choose m(\lambda \setminus i)} \left[(i-2) \prod_{j\in \lambda \setminus i} j{-}1\right] e_{\lambda,i}.
\end{equation*}
Writing $[e_{\lambda,i}]f$ for the coefficient of $e_{\lambda,i}$ in $f \in \Lambda[t]$, we see
\begin{align*}
    [e_{\lambda,i}]X_{C_n,n} &= [e_{\lambda,i}](X_{P_n,n} - tX_{C_{n-1},{n-1}})\\
    &= {\ell(\lambda \setminus i) \choose m(\lambda \setminus i)} \left[\prod_{j \in \lambda \setminus i} j{-}1 + (i-2)\prod_{j \in \lambda \setminus i} j{-}1 \right]\\
     & = {\ell(\lambda \setminus i) \choose m(\lambda \setminus i)} \prod_{j \in \lambda} j{-}1,
\end{align*}
and the theorem follows.
\end{proof}

\subsection{Unit interval graphs} 
If $S$ is a finite set of bounded open intervals in $\RR$, the associated \emph{interval graph} is the graph with vertex set $S$ and an edge from $I_1$ to $I_2$ if $I_1 \cap I_2 \neq \emptyset$. If the intervals all have length $1$ then the graph is a \emph{unit interval graph}. The unit interval graphs are exactly the incomparability graphs of the $(\mathbf{3}+\mathbf{1})$- and $(\mathbf{2}+\mathbf{2})$-free posets (a poset being $(\mathbf{2}+\mathbf{2})$-free if it does not contain the disjoint union of two chains of size two as an induced subposet).

Guay-Paquet \cite{guay-paquet} showed that if $G$ is the incomparability graph of a $(\mathbf{3}+\mathbf{1})$-free poset, then $X_G$ is a convex combination of $X_H$ for some unit interval graphs $H$. In particular, to prove the $e$-positivity of incomparability graphs of $(\mathbf{3}+\mathbf{1})$-free posets conjectured by Stanley and Stembridge, it would suffice to prove it for unit interval graphs.

Suppose $m = m_1 \cdots m_n$ is a sequence of integers $1 \leq m_1 \leq \cdots \leq m_n \leq n$ such that $m_i \geq i$ for each $i$. Associate to $m$ the graph $G(m)$ on $[n]$ with an edge $(i,j)$ if and only if $i < j \leq m_i$. Up to isomorphism, the unit interval graphs on $n$ vertices are exactly the graphs $G(m)$.

\begin{conj} $X_{G(m),1}$ is pointed $e$-positive. \end{conj}           
By Lemma~\ref{lem:psi-e}, this conjecture would imply the $e$-positivity of $X_{G(m)}$ and hence the Stanley-Stembridge conjecture. It holds for $n \leq 7$. One can use the deletion-contraction recurrence to obtain various recurrences for $X_{G(m),1}$, but these recurrences do not a priori preserve pointed $e$-positivity, and the ones we have found have the drawback of either requiring us to broaden our class of graphs, or of passing through some $X_{G(m),v}$ for $v \neq 1$. It is worth noting that $G(m)$ is a chordal graph, so the tools of \S \ref{sec:chordal} apply, although we do not know whether these tools are useful for understanding the $e$-expansion of $X_{G(m)}$.

\section{Acknowledgements}
I would like to thank John Stembridge and Jair Taylor for helpful conversations and comments, and Sara Billey for pointing me to a conjecture of Richard Stanley \cite{stanley-MO-question} which led to some of the ideas here.

\bibliographystyle{plain}
\bibliography{../../bib/algcomb}

\end{document}